\documentclass[11pt,reqno]{amsart}

\usepackage{amssymb}
\usepackage{amsmath,amsthm}
\usepackage{graphicx}
\usepackage[pdftex,unicode,plainpages=false]{hyperref}
\usepackage{lmodern,xfrac}

\theoremstyle{plain}

\newtheorem{prop}{Proposition}
\newtheorem{theor}{Theorem}
\theoremstyle{definition}
\newtheorem{exam}{Example}
\newtheorem{ass}{Assumptions}

\newcommand{\e}{\mathrm{e}}

\newcommand{\R}{\mathbb{R}}
\newcommand{\N}{\mathbb{N}}

\begin{document}

\title{Ornstein-Uhlenbeck approximation of one-step processes: a differential equation approach}

\author{E.~Sikolya}
\address{Department of Applied Analysis and Computational Mathematics\\
E\"otv\"os Loránd University\\
Budapest, Hungary}
\email[E.~Sikolya]{seszter@cs.elte.hu}

\author{P.~L.~Simon}
\address{Department of Applied Analysis and Computational Mathematics\\
E\"otv\"os Loránd University\\
Budapest, Hungary}
\email[P.~L.~Simon]{simonp@cs.elte.hu}

\date{}


\begin{abstract}
The steady state of the Fokker-Planck equation corresponding to a density dependent one-step process is approximated by a suitable normal distribution. Starting from the master equations of the process, written in terms of the time dependent probabilities, $p_k(t)$ of the states $k=0,1,\ldots , N$, their continuous (in space) version, the Fokker-Planck equation is formulated. This PDE approximation enables us to create analytic approximation formulas for the steady state distribution. These formulas are derived based on heuristic reasoning and then their accuracy is proved to be of order $1/N^{\beta}$ with some power $\beta <1$.
\end{abstract}

\maketitle

\section{Introduction}

Deterministic limits and diffusion approximations of density dependent Markov processes have been widely studied since the early works of Kurtz and Barbour \cite{Barbour, Kurtz71}. In these pioneering papers a functional law of large numbers and a central limit theorem were established, claiming that a density dependent process converges (uniformly in probability) over any finite time interval to the solution of the deterministic mean-field ODE model and providing a PDE diffusion approximation for the fluctuations of the process around the deterministic trajectory. These results were put later in a unified context in the framework of martingale theory \cite{EthierKurtz}. The approximation results were motivated by and applied to stochastic population models \cite{Pollett,Ross} and network processes \cite{BallNeal, Barratetal, Danon, Nekovee}. Our main motivation is SIS epidemic propagation on a random graph when the state space is $\{0,1, \ldots , N\}$ with $N$ denoting the number of nodes in the network, and $p_k(t)$ is the probability that there are $k$ infected nodes at time $t$. The process can be described by a density dependent Markov chain with possible transitions from state $k$ to $k-1$ with recovery and to state $k+1$ with infection. The probabilities $p_k(t)$ are determined by a system of linear differential equations, called master equations. Solving this system (with a given initial condition) yields the full description of the process enabling us to view the problem from a differential equation perspective.

New approaches for deriving deterministic limits and diffusion approximations, based purely on differential equation techniques, has been developed recently in \cite{BKS15, BKSS12, KunszSimon16, SK10}. In \cite{BKSS12} it is shown by using the approximation theory of operator semigroups that the difference between the expected value $\sum \frac{k}{N}p_k(t)$ and the mean-field approximation is of order $1/N$. This operator semigroup approach enabled the authors to approximate not only the expected value but also the distribution $p_k$ itself using a partial differential equation in \cite{BKS15}. The approximation is based on introducing a two-variable function $u$ for which $u(t, k/N) \approx p_k(t)$ and deriving the Fokker-Planck equation \cite{risken2012fokker}. Then the master equation can be considered to be the discretisation of the Fokker-Planck equation in an appropriate sense. Armbruster and co\-workers developed a simple approach in \cite{Armbruster}, based only on elementary ODE and probability tools, to prove that the accuracy of the mean-field approximation is order $1/N$, providing also lower and upper bounds for the expected value that can be used for finite $N$ (in contrast to the asymptotic results).

According to \cite{Ross}, the diffusion approximation can be strengthened by identifying an approximating Ornstein-Uhlenbeck process. Our main focus in this paper is on the approximation of the stationary solution of the Fokker-Planck equation by a normal distribution. This can be carried out by approximating the Fokker-Planck equation with a parabolic PDE, in which the drift coefficient is linear and the diffusion coefficient is constant, hence it corresponds to an Ornstein-Uhlenbeck process, see Section 5.3 in \cite{risken2012fokker}. The solution of this approximating PDE can be given explicitly as a normal distribution, moreover, we can prove by using only elementary differential equation techniques that the difference between the stationary solutions of the Fokker-Planck equation and its approximation is of order $1/N^{\beta}$ with some power $\beta <1$.

Although the problem can be formulated in very general terms, here we restrict ourselves to a specific situation which creates a balance between tractability and mathematical generality. We make the following three assumptions. First, the process is assumed to be Markovian and density dependent as it is defined in \cite{Ross}. The state space is then a subset of $\mathbb{Z}^D$, and our second assumption is that $D=1$ with the state space chosen as $\{0,1, \ldots , N\}$. Finally, we assume that the transition from state $k$ is possible only to states $k-1$ and to $k+1$, i.e. only one-step processes are considered (called also counting or birth-death processes). The second and third assumptions are mainly technical, i.e. the proof is probably extendable to the general density dependent case. We note that in our case the transition matrix is tridiagonal, hence powerful methods, e.g. that developed recently by Smith and Shahrezaei \cite{SmithShahrezaei} can be used for computational purposes. However, here our goal is the theoretical approximation of the steady state distribution, which is given by the eigenvector corresponding to the zero eigenvalue of the transition matrix. This is approximated by the steady state solution of the Fokker-Planck equation. In the special case, when the transition rates depend linearly on $k$, the coordinates of the eigenvector are given by a binomial distribution and the steady state of the Fokker-Planck equation is a normal distribution, hence our approximation result reduces to the Moivre-Laplace theorem. We will prove that a similar result holds in the nonlinear case as well. A novelty of our result is that it is formulated in differential equation terms and its proof uses only elementary analysis techniques. Hence it may be reachable for a broader part of the scientific community, including those who are more familiar with differential equations than stochastic techniques.

The paper is structured as follows. The problem setting is formulated in Section \ref{sec:setting}. Then, as a motivation for the further study, the approximation result is presented in the case, when the transition rates depend linearly on $k$ in Section \ref{sec:lincoeff}. Our main general approximation result is formulated in Section \ref{sec:general} and proved in Section \ref{sec:proof}. In Section \ref{sec:discussion} we give a brief outlook to further results on time dependent solutions.

\section{Setting of the problem} \label{sec:setting}

Consider a continuous time Markov chain with state space $\{0,1, \ldots , N\}$. Denoting by $p_k(t)$ the probability of state $k$ at time $t$ and assuming that transition from state $k$ is possible only to states $k-1$ and $k+1$, the \emph{master equation} of the process takes the form
\begin{equation*}\tag{ME}
    \dot{p}_k= a_{k-1}p_{k-1}-(a_{k}+c_{k})p_k+c_{k+1}p_{k+1},\quad k=0,\ldots,  N .
\end{equation*}
The equation corresponding to $k=0$ does not contain the first term in the right hand side, while that corresponding to $k=N$ does not contain the third term, i.e. there are no terms belonging to $a_{-1}$ and to $c_{N+1}$. Moreover, in order to have a proper Markov chain, where the sum of each coloumn in the transition matrix is zero, we assume that $a_N=0=c_0$.

Several network processes can be described by this prototype model. For example, in the case of $SIS$ propagation on a complete graph, or on a configuration random graph $p_k(t)$ is the probability that there are $k$ infected nodes. For a complete graph $a_k=\tau k(N-k)$, $c_k=\gamma k$, where $\tau=\beta/N$ is the rate of infection across an edge and $\gamma$ is the rate of recovery of a node. (It is important here that the infection rate $\tau$ scales with $1/N$ because otherwise the infection pressure to a node would tend to infinity as the number of nodes, together with the degree of a node, tend to infinity.) For configuration random graphs with different degree distributions, e.g. regular random graphs and power-law graphs, the coefficient $a_k$ was determined numerically from simulations in \cite{akckcikk}.

The infinite size limit, i.e. the case when $N\to \infty$, can be described by differential equations in the so-called density dependent case, when the transition rates $a_k$ and $c_k$ can be given by non-negative, continuous functions $A,C:[0,1]\to [0,+\infty)$ satisfying $A(1)=0=C(0)$ as follows
\begin{equation} \label{eq:dendep}
\frac{a_k}{N} = A\left( \frac{k}{N}\right)  \quad \mbox{ and } \quad \frac{c_k}{N} = C\left( \frac{k}{N}\right).
\end{equation}
We note that the conditions $A(1)=0=C(0)$ ensure $a_N=0=c_0$. The special case when these functions are linear or constant can be fully described mathematically, and will serve as motivation for studying the nonlinear case. We note that this definition is the special case of Definition 3.1 in \cite{Ross}.

\subsection{Deterministic limit: mean-field equation}

Once the above system is solved for $p_k$, we can determine the expected value (first moment) as
\begin{equation}
m_1(t) = \sum_{k=0}^N \frac{k}{N} p_k(t) . \label{y1}
\end{equation}
In the case of epidemic propagation this is the expected proportion of infected nodes at time $t$. In the density dependent case \eqref{eq:dendep} we obtain the following differential equation for $m_1$
\[
\dot m_1 = \sum_{k=0} ^N \left[A\left( \frac{k}{N}\right) - C\left( \frac{k}{N}\right) \right] p_k,
\]
see \cite[Lemma 2]{BKSS12}.
Introducing $y_1$ as the approximation of $m_1$, the approximating closed differential equation -- called \emph{mean-field equation} -- takes the form
\begin{equation*}\tag{MF}
\dot y_1 = A(y_1) -C(y_1).
\end{equation*}
In \cite{BKSS12} it was proved that in a bounded time interval the accuracy of the approximation can be estimated as
\[
\left| m_1(t)-y_1(t)\right| \leq \frac{K}{N},
\]
where $K$ is a constant depending on the length of the time interval.

\subsection{Diffusion approximation: Fokker-Planck equation}

The aim of our investigation in this paper is to approximate the distribution $p_k$ itself. It will be carried out by using a PDE, called \emph{Fokker-Planck equation} \cite{risken2012fokker} that can be considered as the continuous version of the master equation (ME). We wish to approximate the solution $p_k(t)$ by considering it as a discretisation of a continuous function $u(t,z)$ in the interval $[0,1]$, i.e.,
\begin{equation}
u\left(t,\frac{k}{N}\right)=p_k(t)
\label{up}
\end{equation}
for $0\leq k\leq N$. The PDE is usually given in the form
\begin{equation*}\tag{FP}
\partial_t u(t,z) =  \partial_{zz} (g(z)u(t,z)) - \partial_{z} (h(z)u(t,z)) .
\end{equation*}
The functions $g$ and $h$ are determined in such a way that the finite difference discretization of (FP) will yield the master equation (ME). We follow \cite[Section 3]{BKS15} and \cite[Section 2.2]{KunszSimon16}, and use the second order finite difference discretization approximation
\begin{equation}\label{eq:masodrappr}
f(z-h) -2f(z)+f(z+h)\approx  h^2f''(z) , \quad f(z+h)-f(z-h) \approx 2hf'(z)
\end{equation}
for a function $f$ smooth enough.

Thus we get that the desired unknown functions $g$ and $h$ have to be defined in such a way that the relations
\[
g\left(\frac{k}{N}\right) = g_k = \frac{1}{2 N^2} (a_k + c_k) , \quad h\left(\frac{k}{N}\right) = h_k = \frac{1}{N} (a_k - c_k)
\]
hold.

The corresponding boundary conditions are
\begin{equation}\label{eq:bc1}
\partial_z(g u)\left(t,-\frac{1}{2N}\right)-(h u)\left(t,-\frac{1}{2N}\right)=0, \text{ and }
\end{equation}
\begin{equation}\label{eq:bc2}
 \partial_z(g u)\left(t,1+\frac{1}{2N}\right)-(h u)\left(t,1+\frac{1}{2N}\right)=0,
\end{equation}
see \cite[Section 3]{BKS15}.

In the density dependent case \eqref{eq:dendep}, we obtain that $g$ and $h$ can be given as
\begin{equation}\label{eq:FPdenscoeff}
g(z) = \frac{1}{2N} (A(z) +C(z)) , \quad h(z) = A(z)-C(z) .
\end{equation}
Hence, the Fokker-Plank equation for density dependent coefficients is
\begin{equation}\label{eq:FPdens}
\partial_t u(t,z) = \frac{1}{2N} \partial_{zz} ((A(z) +C(z))u(t,z)) - \partial_{z} ((A(z)-C(z))u(t,z))
\end{equation}
subject to boundary conditions
\begin{align}\label{densbc1}
\frac{1}{2N} \partial_z((A+C) u)(t,-\delta)-((A-C) u)(t,-\delta)&=0,\\
\label{densbc2}
\frac{1}{2N} \partial_z((A+C) u)(t,1+\delta)-((A-C) u)(t,1+\delta)&=0,
\end{align}
where $\delta = \frac{1}{2N}$.

\section{Steady state of the Fokker-Planck equation: Linear coefficients}  \label{sec:lincoeff}

If we have linear coefficients in (ME) we obtain special forms for $g$ and $h$, enabling us to determine the steady state solution analytically. Let the coefficients in \eqref{eq:dendep} be given as
\begin{equation}
A(z)=a\cdot(1-z),\quad C(z)=c\cdot z
\label{eq:AClinear}
\end{equation}
with some positive constants $a$ and $c$. Then equation \eqref{eq:FPdens} takes the form
\begin{equation}\label{eq:FPlinear}
\partial_tu(t,z) = \frac{1}{2N} \partial_{zz}\left(((c-a)z+a)u(t,z)\right) - \partial_z\left((a-(a+c)z)u(t,z)\right) .
\end{equation}
The solution in the steady state, i.e. when $\partial_tu(t,z) =0$, will be determined as follows.

\medskip

1. The derivation is carried out first in the special case $a=c$ for the sake of simplicity. Denoting the steady state solution by $U(z)$ it satisfies the ODE
\[
\frac{1}{2N} U''(z) = ((1-2z)U(z))' .
\]
Integrating this equation leads to
\[
\frac{1}{2N} U'(z) = (1-2z)U(z) + K .
\]
The boundary condition \eqref{densbc1} at $z=-1/2N$ implies that $K=0$. Then the equation can be easily integrated again by the separation of the variables yielding
\[
U(z)=U\left(\frac12 \right) \exp\left(-2N(z-\frac12)^2\right).
\]
The constant $U(\frac12)$ has to be chosen in such a way that the integral of $U$ become $1/N$, see \cite[Section 3]{BKS15}. This gives
\begin{equation}
U(z)=\frac{\sqrt{2}}{\sqrt{\pi N}} \exp\left(-2N(z-\frac12)^2 \right). \label{eq:ULinnormal}
\end{equation}
Note that this is an approximation of the binomial distribution. Namely, according to the Moivre-Laplace theorem the binomial distribution
\[
B_k(N,q) = \binom{N}{k} q^k (1-q)^{N-k}
\]
can be approximated by the normal distribution as
\begin{equation}\label{eq:MoivreLap}
B_k(N,q) \approx \frac{1}{\sqrt{Nq(1-q)}} \phi\left( \frac{k-Nq}{\sqrt{Nq(1-q)}} \right),
\end{equation}
where
\[
\phi(x) = \frac{1}{\sqrt{2\pi}} \exp\left( -\frac{x^2}{2} \right)
\]
is the density function of the standard normal distribution. Applying this approximation for $q=1/2$ yields that
\[
B_k(N,1/2) \approx U(k/N)
\]
that is the steady state of the Fokker-Planck equation can be considered as the continuous version of the binomial distribution, which is the steady state of the master equation (ME). The accuracy of the Fokker-Planck equation is illustrated in the left panel of Figure \ref{fig:Linss}, where the exact steady state of the master equation is plotted together with function $U$. One can see that the agreement is excellent even for $N=50$.

\begin{figure}
 \centering
 \includegraphics[scale=0.35]{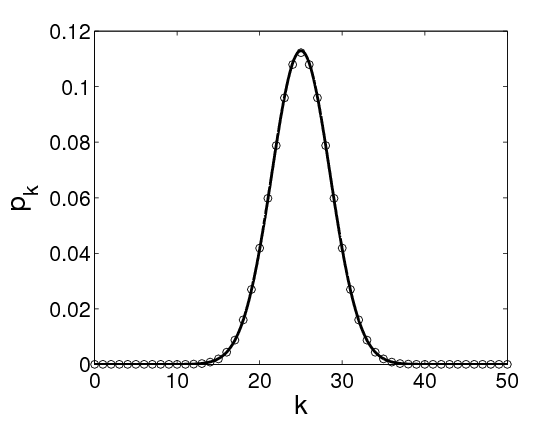}  \includegraphics[scale=0.35]{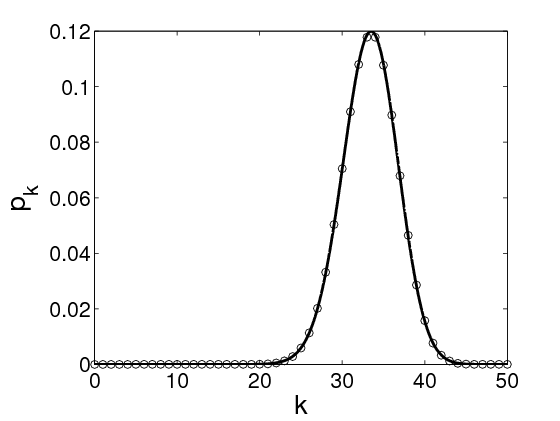}
         \caption{The steady state of the distribution in the linear case, when $A(z)=a(1-z)$ and $C(z)=cz$ for $N=50$. The binomial distribution as the exact solution of the master equation (circles) is shown together with $U$, the solution of the Fokker-Planck equation (continuous curve). In the left panel the case $a=c=1$ is shown, when $U$ is given by \eqref{eq:ULinnormal}. In the right panel the case $a=2$, $c=1$ is shown, when $U$ is given by \eqref{ULingeneral}. }  \label{fig:Linss}
\end{figure}

\medskip

2. In the general case when $a=c$ is not assumed, the stationary solution $U$ satisfies the differential equation
\begin{equation}\label{eq:FPlinstat}
\frac{1}{2N} (((c-a)z+a)U(z))'' = (((a-(a+c)z)U(z))' .
\end{equation}
Integrating this equation leads to
\[
\frac{1}{2N} (((c-a)z+a)U(z))' = (a-(a+c)z)U(z)
\]
since the integrating constant is zero due to the boundary condition. This differential equation can be solved by separation of variables, yielding
\begin{equation}
U(z)=\frac{K}{(c-a)z+a} \e^{H(z)},\label{ULingeneral}
\end{equation}
where the constant $K$ is determined in such a way that the integral of $U$ is $1/N$, see \cite[Section 3]{BKS15}, and
\begin{equation}\label{eq:Hlineset}
H(z)= \frac{2N}{(a-c)^2}\left[ (a^2-c^2)z +2ac\ln(a+(c-a)z) \right] .
\end{equation}
This steady state does not coincide with a normal distribution, however, it will be shown in Subsection \ref{subsec:approx_steadystate} that it can be easily approximated by a normal distribution which is close to the corresponding binomial distribution. This is illustrated in the right panel of Figure \ref{fig:Linss}, where function $U$ is plotted together with the binomial distribution.

\section{Steady state of the Fokker-Planck equation: General case} \label{sec:general}

Introducing the differential operators
\begin{equation}
Dv:= (gv)'-hv \quad \mbox{ and } \quad  Lv:=(Dv)',\label{eq:FPdiffopL}
\end{equation}
the Fokker-Planck equation (FP) takes the form
\begin{equation}\label{eq:FPdiffop}
\partial_t u=Lu.
\end{equation}
The boundary conditions can be written as
\begin{equation}
(Du)(\alpha)=0,\quad (Du)(\beta)=0
\label{eq:bcFPdiffop}
\end{equation}
with $\alpha=-1/2N$, $\beta=1+1/2N.$ Then simple integration shows that $\int_{\alpha}^{\beta}u(t,z)\, dz$ is constant in time. According to \cite[Section 3]{BKS15}, this constant should be equal to $1/N$.

In general, we can say that the Fokker-Plank equation is a parabolic PDE with given initial and boundary conditions. Hence, its solution can be given by using the Fourier method.
We obtain that the solution of \eqref{eq:FPdiffop} subject to the boundary conditions \eqref{eq:bcFPdiffop} can be given as
\[
u(t,z)= \sum_{k=0}^{\infty} c_k \e^{\lambda_k t}v_k(z)
\]
with coefficients $c_k$ determined by the initial condition
\[
u_0(z)= \sum_{k=0}^{\infty} c_k v_k(z) .
\]
The eigenvalue problem belonging to the Fokker-Planck equation can be transformed to a Sturm-Liouville problem, hence it has countably many eigenvalues and its eigenfunctions form a complete system, see \cite{risken2012fokker} p. 106.

\subsection{Stationary solution}

The eigenvalues cannot be determined explicitly in general, hence we consider only the stationary solution. If $\lambda_0 = 0$, then there is a stationary solution $v$ satisfying
\[
Lv = 0, \qquad (Dv)(\alpha) = 0 = (Dv)(\beta) .
\]
The definition of $L$ in \eqref{eq:FPdiffopL} yields that $Dv$ is constant, when $v$ is the stationary solution. According to the boundary condition this constant is zero, i.e. $Dv=0$, yielding $(gv)'=hv$. Introducing the function $f=gv$, this differential equation is equivalent to $f'=hf/g$. This can be integrated to yield $f(z)=K\cdot\exp(H(z))$, where $K$ is a constant and $H$ is the primitive of $h/g$, i.e. $H'=h/g$. Thus the stationary solution is
\begin{equation}\label{eq:stacmo}
v(z)=\frac{K}{g(z)}\e^{H(z)}, \quad \mbox{ with } \quad H'(z)=\frac{h(z)}{g(z)} ,
\end{equation}
where the constant $K=K(N)$ is determined by $\int_{\alpha}^{\beta} v(z)\, dz = \frac1N$.

In the following we turn our attention to the density dependent case and show a method of approximating the stationary solution.

Using \eqref{eq:FPdenscoeff}, the stationary solution \eqref{eq:stacmo} of the Fokker-Planck-equation in the density dependent case has the form
\begin{equation}\label{eq:stacmodens}
v(z)=\frac{2NK}{A(z)+C(z)}\e^{H(z)} \quad \mbox{ with } \quad H'(z)=\frac{2N(A(z)-C(z))}{A(z)+C(z)}.
\end{equation}

Hence, the integral of the function $\frac{A-C}{A+C}$ is needed. If $A$ and $C$ are polynomials then this integral can be explicitly determined, however, the formulas become rather complicated even for low degree polynomials. The case of first order polynomials, when $A(z)=a(1-z)$ and $C(z)=cz$, was solved in Section \ref{sec:lincoeff}.
Then computing the integral of the function $\frac{a-(a+c)z}{a+(c-a)z}$ yields the formula \eqref{eq:Hlineset}, leading to a rather complicated formula for the stationary solution $v$.

\subsection{Approximation of the stationary solution} \label{subsec:approx_steadystate}

A significantly simpler approximation, with normal distribution, can be derived by using the \emph{Ornstein-Uhlenbeck approximation} corresponding to the case when the drift coefficient is linear and the diffusion coefficient is constant, see Section 5.3 in \cite{risken2012fokker}. This uses the linear approximation of the coefficient functions $A-C$ and $A+C$. The approximation is based on the observation that the stationary distribution is concentrated around its expected value, which can be approximated by the steady state solution of the mean-field equation (MF). This steady state is the solution $z^*\in [0,1]$ of the equation
\[A(z^*)-C(z^*)=0,\]
which exists because of the sign conditions $A(0)\geq 0$, $C(0)=A(1)=0$ and $C(1)\geq 0$. Then the following zeroth order approximation is used in the diffusion term
\[A(z)+C(z)\approx A(z^*)+C(z^*)\]
and the first order approximation below is applied in the drift term
\[A(z)-C(z)\approx (A'(z^*)-C'(z^*))(z-z^*).\]
Thus
\[
\frac{A(z)-C(z)}{A(z)+C(z)} \approx \frac{(A'(z^*)-C'(z^*))(z-z^*)}{A(z^*)+C(z^*)} ,
\]
the integral of which is a quadratic function. Hence, the function $H$ in \eqref{eq:stacmodens} is approximated as
\begin{equation}\label{eq:Hfvkoz}
H(z)\approx \overline{H}(z):=Nq (z-z^*)^2  \quad \mbox{ with } \quad q=\frac{A'(z^*)-C'(z^*)}{A(z^*)+C(z^*)} .
\end{equation}
Using this, the stationary distribution \eqref{eq:stacmodens} can be approximated by the normal distribution
\begin{equation}\label{eq:stacmokoz}
v(z)\approx w(z)=K_1\e^{\overline{H}(z)}=K_1\e^{Nq (z-z^*)^2}.
\end{equation}

The first question is how the constant $K_1$ in \eqref{eq:stacmokoz} should be chosen. One idea is that $K_1$ should ensure -- as for $v$ -- that $\int_{\alpha}^{\beta} w(z)\mbox{d}z = \frac1N$. But it turns out that for our purposes the following method is more expedient. Let us take the constant $K=K(N)$ and the primitive function $H$ in \eqref{eq:stacmodens} such that $H(z^*)=0$ which is a natural assumption since in \eqref{eq:Hfvkoz} we approximate $H$ by a function that is $0$ in $z^*.$  This means that
\begin{equation}
H(z)=2N\int_{z^*}^{z}\frac{A(x)-C(x)}{A(x)+C(x)}\, dx=:NB(z)
\label{eq:Hform}
\end{equation}
with
\begin{equation}
B(z)=2\int_{z^*}^{z}\frac{A(x)-C(x)}{A(x)+C(x)}\, dx.
\label{eq:Bform}
\end{equation}
Then let
\[
K_1:=\frac{2NK}{A(z^*)+C(z^*)}
\]
ensuring that
\[
v(z^*)=w(z^*).
\]
Hence,
\begin{equation}\label{eq:wform}
w(z)=\frac{2NK}{A(z^*)+C(z^*)}\e^{Nq (z-z^*)^2}=:\frac{2NK}{A(z^*)+C(z^*)}\e^{Np(z)}
\end{equation}
with
\begin{equation}
p(z)=q (z-z^*)^2.
\label{eq:p}
\end{equation}

\begin{exam} \label{example1}

In the linear case, when $A(z)=a(1-z)$ and $C(z)=cz$, the solution of the equation $A(z)-C(z)=0$ is $z^*=\frac{a}{a+c}$ and $q=-\frac{(a+c)^2}{2ac}$. Hence using \eqref{eq:stacmodens} and \eqref{eq:Hform}, the exact formula for the steady state is
\[
v(z)=\frac{2NK}{a+(c-a)z} \e^{H(z)} ,
\]
where
\[
H(z) = \frac{2N}{(a-c)^2}\left[ (a^2-c^2)(z-z^*) +2ac\ln\frac{a+(c-a)z}{a+(c-a)z^*} \right]
\]
and $K$ is the normalization constant given by the equation $\int_{\alpha}^{\beta} v(z)\, dz = \frac1N$. According to \eqref{eq:wform} the approximating formula for the steady state takes the form
\[
w(z)=\frac{NK(a+c)}{ac} \exp \left( -N\frac{(a+c)^2}{2ac}(z-z^*)^2 \right) .
\]
If $a$ and $c$ are of the same magnitude, then  $w$ yields an extremely accurate approximation of the exact solution $v$, in fact they are visually indistinguishable if plotted in the same figure. In order to show the difference between them they are plotted for $a=10$ and $c=1$ in Figure \ref{fig:LinssNormAppr} together with the steady state of the master equation, which is a binomial distribution with parameter $a/(a+c)$.

\begin{figure}[!ht]
 \centering
 \includegraphics[scale=0.5]{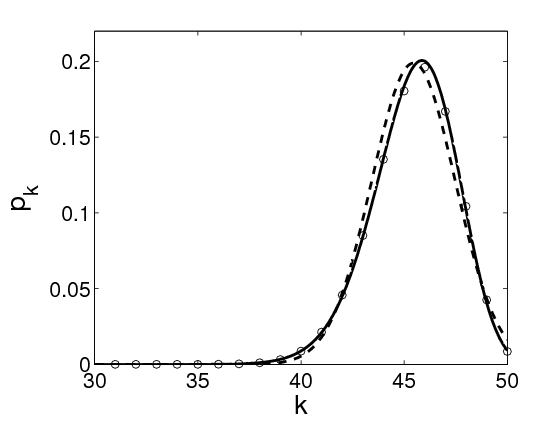}
        \caption{The steady state of the distribution in the linear case, when $A(z)=a(1-z)$ and $C(z)=cz$. The binomial distribution as the exact solution of the master equation (circles) is shown together with $v$, the solution of the Fokker-Planck equation (continuous curve), and with $w$, the solution of the approximate Fokker-Planck equation (dashed curve). The parameter values are $N=50$, $a=10$ and $c=1$.}  \label{fig:LinssNormAppr}
\end{figure}
\end{exam}

Our purpose is now to give exact bounds on the accuracy of the steady state approximation $w$ given by \eqref{eq:wform} in the general density dependent case. That is, we will estimate the distance of the functions
\begin{equation}
v(z)=\frac{2NK}{A(z)+C(z)}\e^{NB(z)}\quad \text{ and } \quad w(z)=\frac{2NK}{A(z^*)+C(z^*)}\e^{Np(z)}
\label{eq:vw}
\end{equation}
for $z\in[0,1]$, where $B(z)$ is given in \eqref{eq:Bform}, $p(z)$ is given in \eqref{eq:p} and $z^*$ is the solution of the equation $A(z^*)-C(z^*)=0$.
It turns out that if $N$ is large enough, the difference is $\mathcal{O}(N^{-\beta})$ for any $0<\beta<1$. In order to formulate this statement rigorously, we collect the assumptions about the coefficient functions $A$ and $C$.
\begin{ass} \label{ass:main}
The functions $A$ and $C$ are assumed to satisfy the following conditions.
\begin{description}
	\item[a1] $A\in C^2[0,1]$, $C\in C^2[0,1]$ are nonnegative functions;
	\item[a2] $A+C>0$ on $[0,1]$ and $z^*$ is the unique root of $A-C$ in $[0,1]$;
	\item[a3] $A'(z^*)-C'(z^*)<0$.
\end{description}
\end{ass}
These assumptions imply that $z^*\in (0,1)$ and $A-C$ is positive on $[0,z^*)$ and negative on $(z^*,1]$. Furthermore, $q$ given in \eqref{eq:Hfvkoz}, is a negative number. Assumptions \ref{ass:main} also imply that $z^*$ is a globally stable equilibrium point of the mean-field equation
\begin{equation*}\tag{MF}
\dot{z}=A(z)-C(z)
\end{equation*}
in $[0,1].$
We are now in the position to formulate our main result.
\begin{theor}\label{thm:main}
Let $v$ be the steady state of the Fokker-Planck equation and $w$ be its approximation given by \eqref{eq:vw}. Let the coefficient functions $A$ and $C$ satisfy Assumptions \ref{ass:main}. Then for each $0<\beta<1$ there exist $C>0$ and $N_0=N_0(\beta,C)\in\N$ such that
\begin{equation}
|v(z)-w(z)|\leq\frac{C}{N^{\beta}},\quad N\geq N_0,\quad z\in[0,1].
\end{equation}
\end{theor}
The proof contains some rather technical tools and is expounded in the next section.

We conclude this section by summarizing our approximation result in order to ease its practical application. Our goal was to approximate the steady state distribution of the master equation (ME), which is given by the leading eigenvector of the matrix in the right hand side of (ME). This distribution is approximated by the steady state $v$ of the Fokker-Planck equation given in \eqref{eq:stacmodens}. This function may be difficult to determine explicitly because of the integration in the definition \eqref{eq:Bform} of $B(z)$, hence a further approximation is given by $w$, as a normal distribution. That is the steady state distribution $p_k$ of the master equation (ME) can be approximated as $p_k\approx w(k/N)$ by using the explicit formula of $w$ in \eqref{eq:wform}. The accuracy of this approximation was illustrated by Example \ref{example1} and is formulated in rigorous terms in Theorem \ref{thm:main}.

\section{Proof of the main theorem} \label{sec:proof}

To prove Theorem \ref{thm:main} we will use the Taylor expansion of function $B$ given in \eqref{eq:Bform}. Since $A(z^*)=C(z^*)$ and
\begin{align*}
B'(z)&=2\frac{A(z)-C(z)}{A(z)+C(z)},\\
B''(z)&=2\frac{(A'(z)-C'(z))(A(z)+C(z))-(A(z)-C(z))(A'(z)+C'(z))}{(A(z)+C(z))^2},
\end{align*}
we obtain that together with
\begin{equation}\label{eq:Bpzcsillag1}
B(z^*)=p(z^*)=0,\quad B'(z^*)=p'(z^*)=0
\end{equation}
also
\begin{equation}\label{eq:Bpzcsillag2}
B''(z^*)=p''(z^*)=2\frac{A'(z^*)-C'(z^*)}{A(z^*)+C(z^*)}=2q<0
\end{equation}
holds.

The next two propositions serve as important technical tools for the proof of our main theorem.
\begin{prop}\label{prop:segedall1}
There exist negative real numbers $r$ and $R$ with
\[r\leq q\leq R<0\]
such that
\[r(z-z^*)^2\leq B(z)\leq R(z-z^*)^2\text{ for all }z\in[0,1].\]
\end{prop}
\begin{proof}
Consider the function $f:[0,1]\to\R$,
\[
f(z):=
\begin{cases}
\frac{B(z)}{(z-z^*)^2}, & z\neq z^*,\\
q, & z=z^*.
\end{cases}
\]
Using L'Hospital's rule and \eqref{eq:Bpzcsillag2} we obtain
\[\lim_{z\to z^*}\frac{B(z)}{(z-z^*)^2}=\lim_{z\to z^*}\frac{B'(z)}{2(z-z^*)}=\lim_{z\to z^*}\frac{B''(z)}{2}=q.\]
Hence, by Assumptions \ref{ass:main} $f$ is a negative continuous function on $[0,1]$. Taking $r$ as the minimum and $R$ as the maximum of $f$ we obtain the statement.
\end{proof}

\begin{prop}\label{prop:segedall2}
Let $|y|\leq 1$. Then
\[|1-\e^{y}|\leq 2|y|.\]
\end{prop}
\begin{proof}
If $-1\leq y\leq 0$ then
\[|1-\e^y|=1-\e^y\leq -y=|y|\leq 2|y|.\]
If $0<y\leq 1$ then
\[|1-\e^y|=\e^y-1<2y=2|y|\]
where we have used that $x\mapsto \e^x$ is a strictly convex function.
\end{proof}

The proof of Theorem \ref{thm:main} is then carried out in three main steps. First, in Proposition \ref{prop:thmbiz1} we estimate the difference
\[\left|\e^{NB(z)}-\e^{Np(z)}\right|.\]
As the next step, in Proposition \ref{prop:thmbiz2}, we derive an upper bound of the form $L/N^{\gamma}$ with $0<\gamma<1/2$ for the difference
\[\frac{1}{2NK}|v(z)-w(z)|.\]
Finally, in Proposition \ref{prop:thmbiz3} we prove that $NK=\mathcal{O}(N^{-\frac{1}{2}})$ and this will complete the proof of the theorem.

\begin{prop}\label{prop:thmbiz1}
For each $0<\gamma<\frac{1}{2}$ there exist $D>0$ and $M_0=M_0(D,\gamma)\in\N$ such that
\begin{equation}
\left|\e^{NB(z)}-\e^{Np(z)}\right|\leq\frac{D}{N^{\gamma}},\quad N\geq M_0,\quad z\in[0,1].
\end{equation}
\end{prop}
\begin{proof}
Let $0<\gamma<\frac{1}{2}$ be arbitrary and define
\[\alpha:=\frac{\gamma+1}{3}.\]
Then $\frac{1}{3}<\alpha<\frac{1}{2}$.\\
We will prove the statement separately for the cases
\[|z-z^*|<\frac{1}{N^{\alpha}}\quad \text{ and } \quad  |z-z^*|\geq\frac{1}{N^{\alpha}} .\]
Case 1: Let $|z-z^*|<\frac{1}{N^{\alpha}}$ for some fixed $N\in\N$. We will apply the equality
\[\left|\e^{NB(z)}-\e^{Np(z)}\right|=\e^{Np(z)}\cdot\left|1-\e^{N(B(z)-p(z))}\right|.\]
Using Taylor's formula, Assumptions \ref{ass:main}.a1, \eqref{eq:Bpzcsillag1} and \eqref{eq:Bpzcsillag2}, we obtain that for each $z\in [0,1]$ there exists $\bar{z}\in (z^*,z)$ (or $\bar{z}\in (z,z^*)$) such that
\begin{align*}
B(z)-p(z)&=B(z^*)-p(z^*)+(B'(z^*)-p'(z^*))\cdot (z-z^*)+\\
&+\frac{1}{2}(B''(z^*)-p''(z^*))\cdot (z-z^*)^2+\frac{1}{6}(B'''(\bar{z})-p'''(\bar{z}))\cdot (z-z^*)^3\\
&=\frac{1}{6}B'''(\bar{z})\cdot (z-z^*)^3,
\end{align*}
since $p'''(z)=0$, $z\in[0,1]$. Hence, if $z\in[0,1]$, $|z-z^*|<\frac{1}{N^{\alpha}}$ then
\[|N(B(z)-p(z))|=\left|\frac{N}{6}B'''(\bar{z})\cdot (z-z^*)^3\right|\leq \frac{D_1}{6}N^{1-3\alpha}=\frac{D_1}{6N^{\gamma}}\]
with
\[D_1=\max_{z\in [0,1]}|B'''(z)|.\]
Thus, there exists $M_1=M_1(D_1,\gamma)$ such that if $N\geq M_1$ and $|z-z^*|<\frac{1}{N^{\alpha}}$ then
\[|N(B(z)-p(z))|<1.\]
Using Proposition \ref{prop:segedall2} we obtain that for such $N$ and $z$
\[\left|1-\e^{N(B(z)-p(z))}\right|\leq 2N|B(z)-p(z)|\leq \frac{D_1}{3N^{\gamma}}=\frac{D_2}{N^{\gamma}}\]
holds with $D_2=\frac{D_1}{3}$. Hence, if $N\geq M_1$ and $|z-z^*|<\frac{1}{N^{\alpha}}$ then
\[\left|\e^{NB(z)}-\e^{Np(z)}\right|\leq\e^{Np(z)}\cdot\frac{D_2}{N^{\gamma}}\leq \frac{D_2}{N^{\gamma}}\]
since $p(z)\leq 0$ for all $z\in [0,1]$.\\
Case 2: Let $|z-z^*|\geq \frac{1}{N^{\alpha}}$ for some fixed $N\in\N$. Then
\[(z-z^*)^2\geq \frac{1}{N^{2\alpha}}.\]
Using Proposition \ref{prop:segedall1} we know that
\[NB(z)\leq NR(z-z^*)^2\leq RN^{1-2\alpha}\]
since $R$ is negative. Hence there exists $M_2=M_2(\gamma)\in\N$ such that
\[\e^{NB(z)}\leq\e^{RN^{1-2\alpha}}<\frac{1}{2N^{\gamma}},\quad N\geq M_2.\]
Similarly, we obtain that there exists $M_3=M_3(\gamma)\in\N$ such that
\[\e^{Np(z)}=\e^{Nq(z-z^*)^2}\leq \e^{qN^{1-2\alpha}}<\frac{1}{2N^{\gamma}},\quad N\geq M_3.\]
Thus taking $M_4=M_4(\gamma):=\max\{M_2,M_3\}$, if $N\geq M_4$ and  $|z-z^*|\geq \frac{1}{N^{\alpha}}$ then
\[\left|\e^{NB(z)}-\e^{Np(z)}\right|\leq\e^{NB(z)}+\e^{Np(z)}<\frac{1}{N^{\gamma}}.\]
We remark that here the inequality holds for any $\gamma>0$ if $N$ is large enough.\\
Combining cases 1 and 2 yields that there exist $D:=\max\{1,D_2\}>0$ and $M_0=M_0(D,\gamma)=\max\{M_1,M_4\}\in\N$ such that
\begin{equation*}
\left|\e^{NB(z)}-\e^{Np(z)}\right|\leq\frac{D}{N^{\gamma}},\quad N\geq M_0,\quad z\in[0,1].
\end{equation*}
\end{proof}

\begin{prop}\label{prop:thmbiz2}
For each $0<\gamma<\frac{1}{2}$ there exist $D'>0$ and $M_0'=M_0'(D',\gamma)\in\N$ such that
\begin{equation}
\frac{1}{2NK}\left|v(z)-w(z)\right|\leq\frac{D'}{N^{\gamma}},\quad N\geq M_0',\quad z\in[0,1].
\end{equation}
\end{prop}
\begin{proof}
We will benefit from the following triangle inequality:
\begin{align}
\left|\frac{v(z)-w(z)}{2NK}\right|&=\left|\frac{\e^{NB(z)}}{A(z)+C(z)}-\frac{\e^{Np(z)}}{A(z^*)+C(z^*)}\right|\notag\\
&\leq\left|\frac{\e^{NB(z)}-\e^{Np(z)}}{A(z)+C(z)}\right|+\e^{Np(z)}\left|\frac{1}{A(z)+C(z)}-\frac{1}{2A(z^*)}\right|\label{eq:thmbiz2_0}\\
&=\left|\frac{\e^{NB(z)}-\e^{Np(z)}}{A(z)+C(z)}\right|+\frac{\e^{Np(z)}\left|A(z)+C(z)-2A(z^*)\right|}{2A(z^*)\cdot(A(z)+C(z))}\label{eq:thmbiz2_1}
\end{align}
since $A(z^*)=C(z^*)>0.$ By Assumptions \ref{ass:main}.a2 we also have that there exists $d>0$ such that
\begin{equation}\label{eq:thmbiz2_2}
A(z)+C(z)\geq d,\quad z\in [0,1].
\end{equation}
Let $0<\gamma<\frac{1}{2}$ be arbitrary. We will again distinguish the following two cases according to the position of $z\in[0,1]$:
\[|z-z^*|<\frac{1}{N^{\gamma}}\quad \text{ and } \quad |z-z^*|\geq\frac{1}{N^{\gamma}} .\]
Case 1: Let $|z-z^*|<\frac{1}{N^{\gamma}}$ for some fixed $N\in\N$. We will use \eqref{eq:thmbiz2_1} for this part of the proof.
From Taylor's formula we obtain that for each $z\in [0,1]$ there exists $\bar{z}\in (z^*,z)$ (or $\bar{z}\in (z,z^*)$) such that
\[A(z)+C(z)-2A(z^*)=\left(A'(\bar{z})+C'(\bar{z})\right)\cdot (z-z^*).\]
Hence,
\begin{equation}\label{eq:thmbiz2_3}
\left|A(z)+C(z)-2A(z^*))\right|\leq D_1'|z-z^*|
\end{equation}
with
\[D_1'=\max_{z\in[0,1]}|A'(z)+C'(z)|.\]
Let $D>0$ and $M_0=M_0(D,\gamma)\in\N$ be the constants from Proposition \ref{prop:thmbiz1}. Combining \eqref{eq:thmbiz2_1}, \eqref{eq:thmbiz2_2}, \eqref{eq:thmbiz2_3}, and using that $p(z)\leq 0$, we obtain that if $N\geq M_0$ and $|z-z^*|<\frac{1}{N^{\gamma}}$ then
\begin{equation*}
\frac{1}{2NK}\left|v(z)-w(z)\right|\leq\frac{D}{dN^{\gamma}}+\frac{1}{2dA(z^*)}\cdot\frac{D_1'}{N^{\gamma}}.
\end{equation*}
Thus, there exist constants
\[D_2':=\frac{D}{d}+\frac{D_1'}{2dA(z^*)}>0\]
and $M_0=M_0(D_2',\gamma)\in\N$ such that if $N\geq M_0$ and $|z-z^*|<\frac{1}{N^{\gamma}}$ then
\[\frac{1}{2NK}\left|v(z)-w(z)\right|\leq\frac{D_2'}{N^{\gamma}}.\]
Case 2: Let $|z-z^*|\geq\frac{1}{N^{\gamma}}$  for some fixed $N\in\N$. We will use \eqref{eq:thmbiz2_0} for this part of the proof.
Since $q<0$ we have that there exists $M_1'=M_1'(\gamma)\in\N$ such that if $N\geq M_1'$ then
\begin{equation}\label{eq:thmbiz2_4}
\e^{Np(z)}=\e^{Nq(z-z^*)^2}\leq\e^{qN^{1-2\gamma}}\leq \frac{1}{N^{\gamma}}.
\end{equation}
Let $D>0$ and $M_0=M_0(D,\gamma)\in\N$ be the constants from Proposition \ref{prop:thmbiz1}. Combining \eqref{eq:thmbiz2_0}, \eqref{eq:thmbiz2_2} and \eqref{eq:thmbiz2_4}, we obtain that if $N\geq \max\{M_0,M_1'\}:=M_2'$ and $|z-z^*|\geq\frac{1}{N^{\gamma}}$ then
\begin{equation*}
\frac{1}{2NK}\left|v(z)-w(z)\right|\leq\frac{D}{dN^{\gamma}}+\frac{1}{N^{\gamma}}\left(\frac{1}{d}+\frac{1}{2A(z^*)}\right).
\end{equation*}
Thus there exist constants
\[D_3':=\frac{D}{d}+\frac{1}{d}+\frac{1}{2A(z^*)}>0\]
and $M_2'=M_2'(D_3',\gamma)\in\N$ such that if $N\geq M_2'$ and $|z-z^*|\geq\frac{1}{N^{\gamma}}$ then
\[\frac{1}{2NK}\left|v(z)-w(z)\right|\leq\frac{D_3'}{N^{\gamma}}.\]
Cases 1 and 2 complete the proof of the statement.
\end{proof}

\begin{prop}\label{prop:thmbiz3}
For the constant $K=K(N)$ in the function
\[v(z)=\frac{2NK}{A(z)+C(z)}\e^{NB(z)},\quad z\in[0,1]\]
we have that $K=\mathcal{O}(N^{-\frac{3}{2}})$.
\end{prop}
\begin{proof}
In Proposition \ref{prop:segedall1} we have seen that there exist negative real numbers $r$ and $R$ such that
\[r(z-z^*)^2\leq B(z)\leq R(z-z^*)^2\text{ for all }z\in[0,1].\]
We know that
\[\int_{-\infty}^{\infty}\e^{-Nz^2}\, dz=\frac{\sqrt{\pi}}{\sqrt{N}}.\]
Since by Assumptions \ref{ass:main}.a2, $A+C$ is a positive continuous function on $[0,1]$, we obtain that
\[\int_0^1 \frac{1}{A(z)+C(z)}\e^{NB(z)}\, dz=\mathcal{O}\left(\frac{1}{\sqrt{N}}\right),\]
that is the integral can be estimated from below and from above by $C/\sqrt{N}$ for some $C>0$ and $N$ large enough.
In \eqref{eq:stacmo} we assumed for $v$ that $\int_{\alpha}^{\beta} v(z)\, dz = 1/N$. Hence
\[\mathcal{O}\left(\frac1N \right) = \int_0^1 v(z)\, dz=2NK\int_0^1 \frac{1}{A(z)+C(z)}\e^{NB(z)}\, dz ,\]
implying that $K=\mathcal{O}(N^{-\frac{3}{2}})$.
\end{proof}

\textbf{Proof of Theorem \ref{thm:main}.} By Proposition \ref{prop:thmbiz2} we know that for each $0<\gamma<\frac{1}{2}$
\[\max_{z\in [0,1]}\frac{1}{2NK}\left|v(z)-w(z)\right|=\mathcal{O}\left(\frac{1}{N^{\gamma}}\right).\]
Using Proposition \ref{prop:thmbiz3} we obtain that $NK=\mathcal{O}(N^{-\frac{1}{2}})$, hence the statement of the theorem follows.

\section{Discussion}\label{sec:discussion}
The method of Subsection \ref{subsec:approx_steadystate} can be used not only for the steady state solution of the mean-field equation (MF) but also at each time instant. We namely know, that the distribution of the solution of (ME) is concentrated at each time around the solution of (MF), $y_1(t)$. Therefore we can substitute $g$ by its value at $y_1(t)$, and we substitute $h$ by the first term of its Taylor expansion around $y_1(t)$, hence, by a linear term.

That is:
\begin{align*}
g(z) &= \frac{1}{2N} (A(z) +C(z))\approx \frac{1}{2N} (A(y_1(t)) +C(y_1(t)))=:b(t),\\
h(z) &= A(z)-C(z) \notag\\
& \approx A(y_1(t)) -C(y_1(t))+\left(A'(y_1(t)) -C'(y_1(t))\right)\cdot (y_1(t)-z)\\
&=:d(t)+l(t)\cdot (y_1(t)-z).
\end{align*}
Putting this in (FP) yields the following Fokker-Planck equation:
\begin{equation}\label{eq:FPappr}
\partial_t \tilde{u}(t,z) = \partial_{zz} (b(t)\tilde{u}(t,z)) - \partial_{z} (d(t)+l(t)\cdot (y_1(t)-z)\tilde{u}(t,z)).\tag{$\widetilde{\mathrm{FP}}$}
\end{equation}
This is the Fokker-Planck equation of a (one-dimensional) \emph{Ornstein-Uhlen\-beck process}. Our conjecture is that using similar methods as above, it is possible to prove that the solution of ($\widetilde{\mathrm{FP}}$) is near to the solution of (FP), hence to the solution of (ME).



\begin{thebibliography}{99}

\bibitem{Armbruster}
B.~Armbruster, \'A.~Besenyei and  P.~L.~Simon, \emph{Bounds for the expected value of one-step processes}, Commun. Math. Sci., \textbf{14} (2016), 1911--1923.

\bibitem{BallNeal}
F.~Ball and P.~Neal, \emph{Network epidemic models with two levels of mixing}, Math. Biosci., \textbf{212} (2008), 69--87.

\bibitem{Barbour}
A.~D.~Barbour, \emph{On a functional central limit theorem for Markov population processes}, Adv. Appl. Prob. \textbf{6} (1974), 21--39.

\bibitem{Barratetal}
A.~Barrat, M.~Barth\'elemy and A.~Vespignani, ''Dynamical {P}rocesses on {C}omplex {N}etworks,'' Cambridge University Press, Cambridge, 2008.

\bibitem{BKS15}
A.~Bátkai, \'A.~Havasi, R.~Horváth, D.~Kunszenti-Kovács and P.~L.~Simon, \emph{PDE approximation of large systems of differential equations}, Oper. Matrices, {\bf 9} (2015), 147--163.

\bibitem{BKSS12}
A.~Bátkai, I.~Z.~Kiss, E.~Sikolya and P.~L.~Simon, \emph{Differential equation approximations of stochastic network processes: an operator semigroup approach}, Netw. Heterog. Media, {\bf 7} (2012), 43--58.

\bibitem{Danon}
L.~Danon, A.~P.~Ford, T.~House, C.~P.~Jewell, M.~J.~Keeling, G.~O.~Roberts, J.~V.~Ross and M.~C.~Vernon, \emph{Networks and the epidemiology of infectious disease}, Interdiscip. Perspect. Infect. Dis. \textbf{2011}, 2011:28~909.

\bibitem{EthierKurtz}
S.~N.~Ethier and T.~G.~Kurtz, ''Markov {P}rocesses: {C}haracterization and {C}onvergence,'' John Wiley \& Sons Ltd, New York, 2005.

\bibitem{KunszSimon16}
D.~Kunszenti-Kovács and P.~L.~Simon, \emph{Mean-field approximation of counting processes from a differential equation perspective}, Electron. J. Qual. Theory Differ. Equ., to appear.

\bibitem{Kurtz71}
T.~G.~Kurtz, \emph{Limit theorems for sequences of jump Markov processes approximating ordinary differential processes}, J. Appl. Probab., \textbf{8} (1971), 344--356.

\bibitem{akckcikk}
N.~Nagy, I.~Z.~Kiss and P.~L.~Simon, \emph{Approximate master equations for dynamical processes on graphs}, Math. Model. Nat. Phenom., \textbf{9} (2014), 32--46.

\bibitem{Nekovee}
M.~Nekovee, Y.~Moreno, G.~Bianconi and M.~Marsili, \emph{Theory of rumour spreading in complex social networks}, Phys. A, \textbf{374} (2007), 457--470.

\bibitem{Pollett}
P.~K.~Pollett, \emph{On a model for interference between searching insect parasites}, J. Austral. Math. Soc. Ser. B, \textbf{32}  (1990), 133--150.

\bibitem{risken2012fokker}
H.~Risken, ''The {F}okker-{P}lanck {E}quation,'' Springer, Berlin, Heidelberg, 1996.

\bibitem{Ross}
J.~V.~Ross, \emph{A stochastic metapopulation model accounting for habitat dynamics}, J. Math. Biol., \textbf{52} (2006), 788--806.

\bibitem{SK10}
P.~L.~Simon and I.~Z.~Kiss, \emph{From exact stochastic to mean-field ODE models: a new approach to prove convergence results}, IMA J. Appl. Math., \textbf{78} (2013), 945--964.

\bibitem{SmithShahrezaei}
S.~Smith and V.~Shahrezaei, \emph{General transient solution of the one-step master equation in one dimension}, Phys. Rev. E,  \textbf{91} (2015), 062119.


\end{thebibliography}
\end{document}